\documentclass[reqno]{amsart}
\usepackage{color}
\usepackage{amssymb,amsmath,amsthm,amstext,amsfonts}
\usepackage[dvips]{graphicx}
\usepackage{psfrag}
\usepackage{url}

\makeatletter \@addtoreset{equation}{section} \makeatother

\renewcommand\thetable{\thesection.\@arabic\c@table}

\theoremstyle{plain}

\newtheorem{theorem}{Theorem}[section]
\newtheorem{proposition}{Proposition}[section]
\newtheorem{lemma}{Lemma}[section]
\newtheorem{corollary}{Corollary}[section]
\newtheorem{definition}{Definition}[section]
\newtheorem{remark}{Remark}[section]

\newtheorem{question}{Question}[section]

\newcommand{\diam}{\operatorname{diam}}

\begin{document}

\title{Measure theoretic pressure and dimension formula for non-ergodic measures}


\author{Jialu Fang}
\address{School of Mathematical Sciences, Soochow University\\
  Suzhou 215006, Jiangsu, P.R. China}
\email{20154207025@stu.suda.edu.cn}

\author{Yongluo Cao}
\address{School of Mathematical Sciences, Soochow University\\
  Suzhou 215006, Jiangsu, P.R. China}
\address{Center for Dynamical Systems and Differential Equation, Soochow University\\
Suzhou 215006, Jiangsu, P.R. China}
\email{ylcao@suda.edu.cn}

\author{Yun Zhao}
\address{School of Mathematical Sciences, Soochow University\\
  Suzhou 215006, Jiangsu, P.R. China}
  \address{Center for Dynamical Systems and Differential Equation, Soochow University\\
Suzhou 215006, Jiangsu, P.R. China}
\email{zhaoyun@suda.edu.cn}

\thanks{Y. Cao is partially supported by NSFC (11790274,11771317), Y. Zhao is  partially supported by NSFC (11790274, 11871361).}

\date{\today}

\begin{abstract} This paper first studies the measure theoretic pressure of measures that are not necessarily ergodic. We define the measure theoretic pressure of an  invariant measure (not necessarily ergodic) via the Carath\'{e}odory-Pesin structure described in \cite{Pes97}, and show  that this quantity  is equal to the essential supremum of the free energy of the measures in an ergodic decomposition. To the best of our knowledge, this formula is new even for entropy. Meanwhile, we define the measure theoretic pressure in another way by using separated sets, it is showed that this quantity is exactly the free energy if the measure is ergodic. Particularly, if the dynamical system satisfies the uniform separation condition and the ergodic measures are entropy dense, this quantity is still equal to the the free energy even if the measure is non-ergodic. As an application of the main result, we find that the Hausdorff dimension of an invariant measure supported on an average conformal repeller is given by the zero of the measure theoretic pressure of this measure. Furthermore, if a hyperbolic diffeomorphism is average conformal and volume-preserving, the Hausdorff dimension of any invariant measure on the hyperbolic set is equal to the sum of the zeros of measure theoretic pressure restricted to stable and unstable directions.
\end{abstract}



\maketitle

\section{Introduction }
Let  $(X,f)$ be a topological dynamical system, i.e., $f$ is a continuous transformation on a compact metric space $X$. Given a continuous function $\varphi$ on $X$ and an $f$-invariant measure $\mu$, the following quantity
\[
h_{\mu}(f)+\int_X \varphi d\mu
\]
is called the \emph{free energy} of $(f,\varphi)$ with respect to $\mu$, where $h_{\mu}(f)$ is the Kolmogorov-Sinai  entropy of $f$ defined via measurable partitions (see Walters' book \cite{wal82} for more details).

The free energy is an important quantity in the study of dynamical system, e.g., it appears as rate function in the study of large deviations in dynamical system (see \cite{young}), and the well-known variational principle which relates  the topological pressure with free energy (see \cite{wal75}). To give other viewpoints for free energy, using ideas of Katok's entropy formula \cite{kat80},  in \cite{hlz} He et al. showed that the free energy can be regarded as the growth
rate of the  minimal value of a potential on an $(n,\epsilon,\delta)$-spanning set ( the union of these $n$-th Bowen balls centered at the points of this set has  measure at least $1-\delta$). In \cite{Pes97}, using the theory of Carath\'{e}odory-Pesin structure, Pesin gave an explanation for free energy from the dimensional viewpoint. The authors in \cite{chz} extended the works in \cite{hlz} and \cite{Pes97} for a larger class of potentials, e.g., the sub-additive and super-additive potentials, it is useful to estimate the dimension of an ergodic measure supported on non-conformal repellers (see \cite{chz} and \cite{wcz}). In \cite{zhao}, the third author of this paper extended the works in \cite{hlz} and \cite{Pes97} to amenable group actions.

We would like to stress that all the results mentioned above require that the measure to be ergodic. However, we know nothing if the measure is non-ergodic.  In \cite{bw}, Barreira and Wolf studied the Hausdorff dimension and point-wise dimension of invariant measures that are not necessarily ergodic. Namely, for conformal expanding maps and hyperbolic diffeomorphisms on surfaces, they established formulas for the point-wise dimension of an invariant measure in terms of the local entropy and of the Lyapunov exponents. Furthermore, they proved that the Hausdorff dimension of an invariant measure is equal to the essential supremum of the Hausdorff dimensions of the measures in an ergodic decomposition.

 The results described above  indicate that a given global quantity has behind it or in certain case can even be build with the help of a local quantity. Another example is that the Kolmogorov-Sinai  entropy of an invariant measure can be obtained by integrating the Brin-Katok's local entropy (see \eqref{pt-global}). Motivated by Barreira and Wolf's results in dimension theory, following the approach described in \cite{Pes97}, we define the measure theoretic pressure from the dimensional viewpoint. Furthermore, we define the point-wise measure theoretic pressure by combining Brin-Katok's local entropy formula and Birkhoff's ergodic theorem. In this paper, we prove that the measure theoretic pressure of any invariant measure is  the essential supremum of the point-wise measure theoretic pressure (see Proposition \ref{super-pw}), then we show that the measure theoretic pressure of any invariant measure can be regarded as the essential supremum of the free energies of the measures in an ergodic decomposition (see Theorem \ref{mainthm}). We would like to point out that these results are new even for entropy.

 As an application of Theorem \ref{mainthm}, we show that the Hausdorff dimension of an invariant measure (not necessarily ergodic) supported on an average conformal repeller is given by the zero of the measure theoretic pressure of this measure (see Theorem \ref{dim-expand}), which can be regarded as Bowen's equation in the measure theoretic sense. To show this result, we first show that the Hausdorff dimension of an invariant measure on average conformal repellers of $C^1$ maps is equal to the essential supremum of the Hausdorff dimensions of the measures in an ergodic decomposition (see Theorem \ref{dim-dec}), which extends Barreira and Wolf's Hausdorff dimension formula (see \cite[Theorem 10]{bw}) to  $C^1$ maps and non-conformal case.
   Following the same approach, we can show that the Hausdorff dimension of an invariant measure supported on  average conformal hyperbolic sets of a $C^1$ volume-preserving diffeomorphism is given by the sum of the zeros of measure theoretic pressure restricted to stable and unstable directions (see Theorem \ref{dim-ach-mt}).

  This paper also defines the measure theoretic pressure of invariant measures (not necessary ergodic) by using $(n,\epsilon)$-separated sets and $(\delta,n,\epsilon)$-separated sets (see Section \ref{equiv-mtp} for details). For each ergodic measure, we show that the measure theoretic pressure defined in this way is exactly the free energy, which means that this definition of measure theoretic pressure is  respectively equivalent to the ones in \cite{hlz} and \cite{Pes97}. Furthermore, if the dynamical system has the uniform separation property (see Definition \ref{uniform-sep-p}) and the ergodic measures are entropy dense (see Definition \ref{def-eed}), then this kind of measure theoretic pressure is still equal to the free energy even if the measure is non-ergodic.

 This paper is organized as follows. In Section \ref{NP}, we introduce the concepts of point-wise measure theoretic pressure and measure theoretic pressure for any invariant measure. We also define the measure theoretic pressure via separated sets. In Section \ref{AR}, we give some auxiliary results that are useful in the proof of main results in this paper. Section \ref{main-result} contains the statements and detailed proofs of our main results. We first prove that the measure theoretic pressure is the essential supremum of the sum of Brin-Katok's local entropy and the limit of Birkhoff's sum of a continuous function, which helps us to show that the measure theoretic pressure is the essential supremum  of the sum of Kolmogorov-Sinai  entropy and the integral of the potential with respect to the measures in an ergodic decomposition. Meanwhile, we show that the measure theoretic pressure of ergodic measures defined by separated sets is equal to the free energy. In particular, the measure theoretic pressure defined via separated sets is still equal to the free energy even if the measure is non-ergodic, provided the system has the uniform separation property and the ergodic measures are entropy dense. Finally, we show that the zero of measure theoretic pressure is exactly the Hausdorff dimension of an invariant measure on average conformal repellers of a $C^1$ map; the Hausdorff dimension of an invariant measure supported on hyperbolic sets is the sum of the zeros of measure theoretic pressure restricted to stable and unstable directions, provided that the diffeomorphism is average conformal and volume-preserving.

\section{Notation and Preliminaries}\label{NP}
We give the definitions and notation in this section. Throughout this section, let $(X,f)$ be a topological dynamical system (TDS for short), i.e., $f: X\to X$ is a continuous  transformation on a compact metric space $X$ equipped with metric $d$. Let $\mathcal{M}(X)$ denote the space of all Borel probability measures on $X$, and let $\mathcal{M}(X,f)$ and $\mathcal{E}(X,f)$ denote the set of all $f-$invariant respectively,
ergodic $f-$invariant Borel probability measures on $X$.

\subsection{Point-wise measure theoretic pressure}

 For $x, y\in X$, define a dynamical metric as $d_n(x,y)=\max\{d(f^ix, f^iy): 0\le i<n\}$, and let
$B_n(x,\epsilon)=\{y\in X: d_n(x,y)<\epsilon\}$ denote the dynamical ball centered at $x$ of radius $\epsilon$ and length $n$.

 Given a continuous function $\varphi: X\to \mathbb{R}$ and an invariant measure $\mu\in \mathcal{M}(X,f)$, the following quantity
\[
P_\mu(f,\varphi, x):=\lim_{\epsilon\to 0}\liminf_{n\to \infty}
\frac{1}{n}\Big[-\log \mu(B_n(x,\epsilon))+S_n\varphi(x)\Big]
\]
is called the \emph{point-wise measure theoretic pressure} of $\varphi$ at the point $x$ (w.r.t. the measure $\mu$), where $S_n\varphi(x):=\sum_{i=0}^{n-1}\varphi(f^ix)$. It follows from the Brin-Katok's local entropy  formula (see \cite{bk}) and Birkhoff's ergodic theorem (e.g., see \cite{wal82}) that the following limits
\begin{eqnarray}\label{local-formula}
h_{\mu}(x):=\lim_{\epsilon\to 0}\liminf_{n\to \infty}
-\frac{1}{n}\log \mu(B_n(x,\epsilon))~~~~\text{and}~~~~\varphi^*(x):=\lim_{n\to \infty}
\frac{1}{n}S_n\varphi(x)
\end{eqnarray}
exist $\mu$-almost everywhere. Hence,
\begin{eqnarray}\label{pw-formula}
P_\mu(f,\varphi, x)=h_{\mu}(x)+\varphi^*(x),~~~\mu-a.e.~x\in X.
\end{eqnarray} Particularly, the  two limits in \eqref{local-formula} satisfy that
\begin{eqnarray}\label{pt-global}
h_\mu(f)=\int_X h_{\mu}(x) d\mu~~~\text{and}~~~\int_X \varphi d\mu=\int_X \varphi^* d\mu
\end{eqnarray}
where $h_\mu(f)$ is the  Kolmogorov-Sinai  entropy of $f$ (see \cite{wal82} for more details).

\subsection{Measure theoretic pressure defined via Carath\'eodory-Pesin construction}We first follow the approach described in \cite{Pes97} to define the topological pressure on an arbitrary subset. Given a continuous function $\varphi: X\to \mathbb{R}$, a subset $Z\subset X$ and $\alpha\in \mathbb{R}$, let
\[
\begin{aligned}
M(Z,\varphi,\alpha,N,\epsilon)=\inf\Big\{&\sum_i
\exp\bigr(-\alpha n_i+S_{n_i}\varphi(x_i)\bigr):\\
&\bigcup_{i}B_{n_i}(x_i,\epsilon)\supset Z, \, x_i\in X \text{ and } n_i\ge N \text{ for all } i\Big\}.
\end{aligned}
\]
Since $M(Z,\varphi,\alpha,N,\epsilon)$ is monotonically
increasing with $N$, let
\[
m(Z,\varphi,\alpha,\epsilon)
:=\lim_{N\rightarrow\infty}M(Z,\varphi, \alpha,N,\epsilon).
\]
We denote the jump-up point of $m(Z,\varphi,\alpha,\epsilon)$ by
\[
P_{Z}(f,\varphi,\epsilon) =\inf \{ \alpha:
m(Z,\varphi,\alpha,\epsilon)=0 \}
=\sup\{ \alpha: m(Z,\varphi,\alpha,\epsilon)=+\infty\}.
\]

\begin{definition}\label{defPmu*} We call the quantity
\begin{eqnarray*}
P_{Z}(f,\varphi)=\lim_{\epsilon\to 0} P_{Z}(f,\varphi,\epsilon)
\end{eqnarray*}
the \emph{topological pressure} of $(f, \varphi)$ on the set $Z$.
\end{definition}

The above definition is equivalent to the one given by Pesin and Pitskel' \cite{pp} (see also \cite{Pes97} ),  see \cite{cli} for the detailed proof.

Given  $\alpha\in \mathbb{R}$ and $Z\subset X$, define
\[
R(Z,\varphi,\alpha,N,\epsilon)=\inf\Big\{\sum_i
\exp\bigr(-\alpha N+S_N \varphi(x_i) \bigr):\\
\bigcup_{i}B_{N}(x_i,\epsilon)\supset Z,\,x_i\in X\Big\}.
\]
 We set
$$
\begin{aligned}
\underline{r}(Z,\varphi,\alpha,\epsilon)
&=\liminf_{N\rightarrow\infty}R(Z,\varphi,\alpha,N,\epsilon), \\
\overline{r}(Z,\varphi,\alpha,\epsilon)
&=\limsup_{N\rightarrow\infty}R(Z,\varphi,\alpha,N,\epsilon)
\end{aligned}
$$
and define the jump-up points of
$\underline{r}(Z,\varphi,\alpha,\epsilon)$ and
$\overline{r}(Z,\varphi,\alpha,\epsilon)$ as
$$
\begin{aligned}
\underline{CP}_{Z}(f,\varphi,\epsilon)&=\inf\{\alpha:
\underline{r}(Z,\varphi,\alpha,\epsilon)=0\}
=\sup\{ \alpha: \underline{r}(Z,\varphi,\alpha,\epsilon)=+\infty \},\\
\overline{CP}_{Z}(f,\varphi,\epsilon)&=\inf \{ \alpha:
\overline{r}(Z,\varphi,\alpha,\epsilon)=0 \}
=\sup\{ \alpha: \overline{r}(Z,\varphi,\alpha,\epsilon)=+\infty\}
\end{aligned}
$$
respectively.

\begin{definition}
We call the quantities
$$
\underline{CP}_{Z}(f,\varphi)=\lim_{\epsilon\to 0}
\underline{CP}_{Z}(f,\varphi,\epsilon)\,\,\text{and}\,\,
\overline{CP}_{Z}(f,\varphi)=\lim_{\epsilon\to 0}
\overline{CP}_{Z}(f,\varphi,\epsilon)
$$
the \emph{lower} and \emph{upper topological pressures} of $(f,\varphi)$ on the set $Z$.
\end{definition}

 Given an $f$-invariant measure
$\mu$, let
$$
\begin{aligned}
P_{\mu}(f,\varphi,\epsilon)
&=\inf\{P_Z(f,\varphi,\epsilon)\colon\mu (Z)=1\}\\
&=\lim_{\delta\to 0}\inf\{
P_Z(f,\varphi,\epsilon)\colon\mu (Z)\ge 1-\delta\}
\end{aligned}
$$
and then we call the following quantity
\begin{equation*}\label{pressure1}
P_{\mu}(f,\varphi):=\lim_{\epsilon\to 0}P_{\mu}(f,\varphi,\epsilon)
\end{equation*}
the \emph{measure theoretic pressure} of $(f,\varphi)$ with respect to the measure $\mu$.
 Let further
\begin{gather*}
\underline{CP}_{\mu}(f,\varphi,\epsilon)
=\lim_{\delta\to 0}\inf\{\underline{CP}_{Z}(f,\varphi,\epsilon)
\colon\mu (Z)\ge 1-\delta\}, \\
\overline{CP}_{\mu}(f,\varphi,\epsilon)
=\lim_{\delta\to 0}\inf\{\overline{CP}_{Z}(f,\varphi,\epsilon)
\colon\mu (Z)\ge 1-\delta\}.
\end{gather*}
We call the following quantities
\begin{eqnarray*}
\underline{CP}_{\mu}(f,\varphi)=\lim_{\epsilon\to 0}
\underline{CP}_{\mu}(f,\varphi,\epsilon),\ \ \
\overline{CP}_{\mu}(f,\varphi)=\lim_{\epsilon\to 0}
\overline{CP}_{\mu}(f,\varphi,\epsilon)
\end{eqnarray*}
the \emph{lower and upper measure theoretic pressures} of $(f,\varphi)$ with respect to the measure $\mu$.
It is proved in \cite{Pes97} that
\begin{eqnarray}\label{metric-ent-form}
P_{\mu}(f,\varphi)=\underline{CP}_{\mu}(f,\varphi)=\overline{CP}_{\mu}(f,\varphi)=h_{\mu}(f)+\int_X \varphi d\mu
\end{eqnarray}
for any $f$-invariant ergodic measure $\mu$, see \cite[Theorem A]{chz} for a generalization of the previous formula to a larger class of potentials.

\begin{remark} If the potential $\varphi\equiv 0$, we will write $P_\mu(f,0)$ as $E_\mu(f)$, and write $\underline{CP}_{\mu}(f,0)$ and $\overline{CP}_{\mu}(f,0)$ as $\underline{CE}_{\mu}(f)$ and $\overline{CE}_{\mu}(f)$ respectively.
\end{remark}

Clearly, $E_\mu(f)=\underline{CE}_{\mu}(f)=\overline{CE}_{\mu}(f)=h_\mu(f)$ if $\mu\in \mathcal{E}(X,f)$, this result was first established by Bowen \cite{bo73}. But this formula doesn't hold anymore if $\mu\in \mathcal{M}(X,f)\setminus \mathcal{E}(X,f)$ (see \cite{Pes97}).

\subsection{Measure theoretic pressure defined via separated sets}\label{equiv-mtp} Given $\epsilon>0$ and $\delta>0$,  recall that a subset $E\subset X$ is   $(n,\epsilon)$-separated, if for any distinct points $x,\ y\in E$ we have $
d_n (x,y)>\epsilon$.  Furthermore, a subset $A\subset X$ is called $(\delta,n,\epsilon)$-separated, if  any distinct points $x,\ y\in A$ satisfy that \begin{equation*}
\#\{j:d(f^{j}(x),f^{j}(y))>\epsilon,\ 0\leq j\leq n-1\}\geq\delta n.
\end{equation*}

Given $x\in X$ and $n\in \mathbb{N}$, consider the empiric measure at point $x\in X$ as  follows:
\begin{equation*} \mathcal{E}_{n}(x):=\frac{1}{n}\sum_{k=0}^{n-1}\delta_{f^{k}x}.
\end{equation*}
 For each neighborhood $F\subset \mathcal{M}(X)$, put
\begin{equation*}
X_{n,F}:=\{x\in X:\mathcal{E}_{n}(x)\in F\}
\end{equation*}and define
\begin{equation*}
\begin{aligned}
N(F;n,\epsilon)&:=\max\{\#E~|~E\subset X_{n,F}~\text{is}~(n,\epsilon)-\text{separated}\},\\
N(F;\delta,n,\epsilon)&:=\max\{\#A~|~A\subset X_{n,F}~\text{is}~(\delta,n,\epsilon)-\text{separated}\}.
\end{aligned}
\end{equation*}

Given a continuous function $\varphi: X\to \mathbb{R}$,  $\epsilon>0$ and $\nu \in \mathcal{M}(X)$. Let $F\subset \mathcal{M}(X)$ be a neighborhood of $\nu$, put
\begin{equation*}
\overline{SP}_{\nu}(f,\varphi):=\lim_{\epsilon\rightarrow 0}\inf_{F\ni\nu}\limsup_{n\rightarrow\infty}\frac{1}{n}\log P(F;\varphi,n,\epsilon)
\end{equation*}
and
\begin{equation*}
\underline{SP}_{\nu}(f,\varphi):=\lim_{\epsilon\rightarrow 0}\inf_{F\ni\nu}\liminf_{n\rightarrow\infty}\frac{1}{n}\log P(F;\varphi,n,\epsilon)
\end{equation*}
where the infimum is taken over any base of neighborhoods of $\nu$ and
\begin{eqnarray}\label{smp}
P(F;\varphi,n,\epsilon)=\sup\Big\{\sum_{x\in E}e^{S_n\varphi(x)}: E\subset X_{n,F} ~\text{is }~(n,\epsilon)-\text{separated}\Big\}.
\end{eqnarray}
If $\overline{SP}_{\nu}(f, \varphi) = \underline{SP}_{\nu}(f, \varphi)$, we denote this common value as $SP_{\nu}(f, \varphi)$. If we consider $(\delta, n,\epsilon)$-separated set in \eqref{smp}, we write the corresponding quantities as $P(F;\varphi,\delta,n,\epsilon)$, $\underline{SP}_{\nu}'(f,\varphi)$, $\overline{SP}_{\nu}'(f,\varphi)$ and $SP_{\nu}'(f,\varphi)$  respectively.

In Section \ref{main-result}, for a given TDS $(X,f)$ and a continuous function $\varphi$ on $X$, we will prove that both $SP_{\nu}(f, \varphi)$ and $SP_{\nu}'(f,\varphi)$ exist and equal to the free energy of $(f,\varphi)$ provided that $\nu$ is ergodic. Furthermore, if the system $(X,f)$ has uniform separation property and the ergodic measures of $(X,f)$ are entropy dense (which we will recall in below), then both $SP_{\nu}(f, \varphi)$ and $SP_{\nu}'(f,\varphi)$ exist and equal to the free energy of $(f,\varphi)$ for any invariant measure $\nu$ (not necessarily ergodic).

\begin{definition}\label{def-eed}
The ergodic measures of a TDS $(X,f)$ are entropy dense, if for each $\nu\in \mathcal{M}(X,f)$ and each neighborhood $F\subset \mathcal{M}(X)$ of $\nu$, when $h^{*}< h_\nu(f)$, there exists an ergodic measure $\mu\in F$ such that $h^{*}<h_\mu(f)$.
\end{definition}

In \cite[Theorem 2.1]{PS05}, Pfister and Sullivan proved that the ergodic measures of a TDS $(X,f)$ are entropy dense if the system $(X,f)$ has the $g$-almost product property, e.g., all $\beta$-shift have the $g$-almost product property, and hence the ergodic measures of all $\beta$-shifts are entropy dense.

\begin{definition}\label{uniform-sep-p}
We say that the TDS $(X,f)$ has uniform separation property, if for any $\eta > 0$,
there exists $\delta^{*} > 0$, $\epsilon^{*} > 0$ such that for each $\mu\in \mathcal{E}(X,f)$ and each neighborhood $F\subset \mathcal{M}(X)$ of $\mu$,
there exists $n^{*}_{F,\mu,\eta}\in\mathbb{N}$ such that for any $n\geq n^{*}_{F,\mu,\eta}$, \begin{equation*}
N(F;\delta^{*},n,\varepsilon^{*})\geq e^{n(h_\mu(f)-\eta)}.
\end{equation*}
\end{definition}

In \cite[Theorem 3.1]{PS07}, Pfister and Sullivan proved that a TDS $(X,f)$ has the uniform separation property if the system $(X,f)$ is expansive or is asymptotically $h$-expansive.

\subsection{Dimension of invariant measures}In this subsection, we will recall the definition of dimension of a measure. Particularly, we will recall some formula for dimension of measures supported on average conformal repellers and average conformal hyperbolic sets.

 \subsubsection{Dimension formula for average conformal expanding maps}Let $f:M\rightarrow M$ be a $C^1$ map on a $d$-dimensional smooth Riemannian manifold $M$, and let $J\subset M$ be a compact $f$-invariant subset. We denote by
$\mathcal{M}(f|_J)$ and $\mathcal{E}(f|_J)$ the set of all
$f-$invariant measures and the set of all ergodic measures
supported on $J$ respectively.

We say that $f$ is expanding on $J$ and that  $J$ is a repeller of $f$ if
\begin{enumerate}
\item[(1)] there exists an open neighbourhood $U$ of $J$ such that $J=\{x\in U: f^n(x)\in U, ~~\forall n\ge 0\}$;
 \item[(2)] there exist $\kappa>1$ and $C>0$ such that $\|D_xf^n(v)\|\ge C\kappa^n\|v\|$ for all $x\in J$, $v\in T_xM$ and $n\ge 1$, where $\|\cdot\|$ is the norm induced by the Remannian metric on $M$.
\end{enumerate}

For $x\in M$ and $v\in T_xM$, the Lyapunov exponent of $v$ at $x$
is the limit
\begin{eqnarray*}
\chi(x,v)=\lim_{n\to\infty}\frac 1n\log \|D_xf^n(v)\|
\end{eqnarray*}
if the limit exists. Given an $f$-invariant measure $\mu$, by the Oseledec multiplicative ergodic
theorem \cite{ose1}, for $\mu$-almost every point $x$, every
vector $v\in T_xM$ has a Lyapunov exponent, and they can be
denoted by $\lambda_1(x)\leq \lambda_2(x)\leq \cdots \leq
\lambda_d(x)$. Furthermore, if
$\mu$ is ergodic, since the Lyapunov exponents are $f$-invariant, we write the Lyapunov exponents as
$\lambda_1(\mu)\leq \lambda_2(\mu)\leq \cdots \leq
 \lambda_m(\mu)$.

A compact invariant set $J\subset M$ is an \emph{average conformal
repeller} if for any $\mu\in \mathcal{E}(f|_J)$, $\lambda_1(\mu)=
\lambda_2(\mu)= \cdots
 =\lambda_m(\mu)>0$.
For simplicity we denote by $\lambda(\mu)$ the unique Lyapunov
exponent with respect to $\mu$.

\begin{definition}Given a Borel probability measure $\mu$, the lower and upper point-wise dimensions of $\mu$ at a point $x\in M$ are defined respectively by
\[
\underline{d}_{\mu}(x)=\liminf_{r\to 0}\frac{\log \mu(B(x,r))}{\log r}~~\text{and}~~\overline{d}_{\mu}(x)=\limsup_{r\to 0}\frac{\log \mu(B(x,r))}{\log r}
\]
\end{definition}

For any $f$-invariant measure $\mu\in \mathcal{M}(f|_J)$,  it is proved in \cite[Theorem A]{cao} that
\begin{eqnarray}\label{ptdim-acr}
\underline{d}_{\mu}(x)=\overline{d}_{\mu}(x)=\frac{h_{\mu}(x)}{\lambda(x)}~~~\mu-a.e.~~x\in J
\end{eqnarray}
where $\displaystyle{\lambda(x):=\lim_{n\to\infty}\frac{1}{n}\log\|D_xf^n\|}$, by Kingman's sub-additive egodic theorem we know that $\lambda(x)$ is well-defined $\mu$-almost everywhere. Furthermore, if $\mu$ is ergodic, then
\begin{eqnarray}\label{ptdim-acr-erg}
\underline{d}_{\mu}(x)=\overline{d}_{\mu}(x)=\frac{h_{\mu}(f)}{\lambda(\mu)}~~~\mu-a.e.~~x\in J
\end{eqnarray} See \cite{bw} for the formula of point-wise dimension in the case of conformal repellers of $C^{1+\alpha}$ maps.

Next, we briefly recall the Hausdorff dimension of a probability measure. Given a set $Z\subset M$, its {\it Hausdorff dimension} is defined by
$${\dim}_HZ=\inf\{s: \;
\lim_{\epsilon\to 0}\inf_{\diam{\mathcal U} <\epsilon}
\sum_{U\in{\mathcal U}}(\diam U)^s=0\},$$
where ${\mathcal U}$ is an open cover of $Z$ and
$\diam{\mathcal U}=\sup\{\diam U:\ U\in{\mathcal U}\}$.
If $\nu$ is a probability measure on $M$, then the Hausdorff dimension
of the measure $\nu$ is given by
$${\dim}_H\nu
=\inf\bigl\{{\dim}_HZ: \;Z\subset M,\; \nu (Z)=1\;\bigr\}.$$
In \cite{young82}, Young established a useful criterion that if $\underline{d}_\nu(x)\ge C_1$ and $\overline{d}_{\nu}(x)\le C_2$ for $\nu$-a.e. $x$, then $C_1\le\dim_H\nu\le C_2$. In fact, this estimation is valid for other dimensions of a measure, e.g., box dimension, information dimension etc, see \cite{young82} for details. This implies that for any ergodic measure $\nu$ supported on average conformal repellers, one has
\begin{eqnarray}\label{dim-acr-erg}
\dim_H\nu=\frac{h_\nu(f)}{\lambda(\nu)}.
\end{eqnarray}

\subsubsection{Dimension formula for average conformal hyperbolic diffeomorphisms}\label{hyperbolic}Let $f:M\to M$ be a diffeomorphism on a $d$-dimensional smooth Riemannian manifold $M$. Assume that $J$ is a compact $f$-invariant locally maximal hyperbolic set, i.e.,  there exist an open neighborhood $U$ such that $J=\bigcap_{n\in \mathbb{Z}}f^n(U)$, and a continuous splitting of the tangent bundle $T_JM=E^s \oplus E^u$,  and constants $c > 0$ and $\kappa\in(0, 1)$ such that for each $x\in J$ the following properties hold:
\begin{enumerate}
\item[(1)] $D_xf (E^i(x))=E^i(f(x)),~~i=s,u$;
\item[(2)] $\|D_xf^n(v)\|\le c\kappa^n\|v\|$, $\forall v\in E^s(x),~~\forall n>0$; and $\|D_xf^{-n}(v)\|\le c\kappa^n\|v\|$, $\forall v\in E^u(x),~~\forall n>0$.
\end{enumerate}
We say that $f$ is average conformal on $J$, if for each ergodic measure $\nu$ on $J$ one has that $\lambda_1(\nu)=\cdots=\lambda_{d_s}(\nu)<0$ and $\lambda_{d_s+1}(\nu)=\cdots=\lambda_d(\mu)>0$, where $d_s=\dim E^s$ and $d_u:=d-d_s=\dim E^u$. That is, the map has only two different Lyapunov exponents $\lambda_s(\nu)<0$ and $\lambda_u(\nu)>0$ with respect to each ergodic measure $\nu$ on $J$.

Let $\varphi_i(x)=\frac{1}{d_i}\log |det(D_xf | E^i(x))|$ for $i=s,u$, and $\mu$ an $f$-invariant measure on $J$. Since $f$ is average conformal on $J$, for $\mu$-almost every $x\in J$ the following limits are well-defined
\begin{eqnarray}\label{ach-lya}
\lambda_i(x):=\lim_{n\to\infty} \frac 1n \log \|D_xf^n | E^i(x)\|=\lim_{n\to\infty} \frac 1n\sum_{j=0}^{n-1}\varphi_i(f^j(x)),~~i=s,u.
\end{eqnarray}
See Lemma 2.1 in \cite{wwcz} for the proof. The numbers $\lambda_s(x)$ and $\lambda_u(x)$ are respectively the negative and positive
values of the Lyapunov exponent at $x$. Furthermore, if $\mu$ is ergodic then they are constant almost everywhere.

Assume that $f$ is a $C^{1}$ diffeomorphism with a compact $f$-invariant
locally maximal hyperbolic set $J$ on which $f$ is average conformal. Then the following properties hold:
\begin{enumerate}
\item[(1)] for any ergodic measure $\nu$ on $J$, one has
\begin{eqnarray}\label{dim-f-ach-erg}
\dim_H\nu=h_\nu(f)\Big(\frac{1}{\lambda_u(\nu)}-\frac{1}{\lambda_s(\nu)}\Big),
\end{eqnarray}
see  \cite[Corollary 1]{wc}  for detailed description;
\item[(2)] for any invariant measure $\mu$ on $J$, one has
    \begin{eqnarray}\label{dim-f-ach-inv}
    \dim_H\mu=\text{ess} \sup\Big\{ h_\mu(x)\Big(\frac{1}{\lambda_u(x)}-\frac{1}{\lambda_s(x)}\Big):~~x\in J\Big\}
    \end{eqnarray}
    with the essential supremum taken with respect to $\mu$, see \cite[Corollary 2]{wc} for details.
\end{enumerate}

\section{Auxiliary results}\label{AR} In this section, we will provide some preliminary results about the separated sets and free energy.

\begin{lemma}\label{Le2} Let $(X,f)$ be a TDS, and let $\xi$ be a finite measurable partition of $X$ and  $\mu\in \mathcal{M}(X,f)$. Then for any Borel set $A\subset X$ with $0<\mu(A)<1$, we have\begin{equation*}
H(\mu,\xi)\leq\log2+\mu(A)H(\mu(\cdot\mid A),\xi)+\mu(X\setminus A)H(\mu(\cdot\mid X\setminus A),\xi),
\end{equation*}
where $H(\mu,\xi):=-\sum_{A_i\in \xi}\mu(A_{i})\log\mu(A_{i})$ and $\mu(\cdot\mid B)$ denotes the measure $\mu$ subject to $B$.
\end{lemma}

\begin{proof}
See \cite[Theorem 8.1]{wal82} for the proof.
\end{proof}

Next we recall some notations that would be used in the following.
Given a finite set $A$, let $\#A$ denote the cardinality of $A$ and let $\Lambda_{n}:=\{0,1,\cdots,n-1\}$. For each $\mathrm{v}, ~\mathrm{w}\in A^{\Lambda_{n}}$, where
~$\mathrm{v}=(\mathrm{v}_{0},\mathrm{v}_{1},\cdots,\mathrm{v}_{n-1})$, $\mathrm{w}=(\mathrm{w}_{0},\mathrm{w}_{1},\cdots,\mathrm{w}_{n-1})$, define the \emph{Hamming distance} on $A^{\Lambda_{n}}$ as
\begin{equation*}
d_{n}^{H}(\mathrm{v},\mathrm{w}):=\#\{i\in\Lambda_{n}:\mathrm{v}_{i}\neq\mathrm{w}_{i}\}.
\end{equation*}

\begin{lemma}\label{Le1}
Let~$A$~be a finite set, for any $\mathrm{w}\in A^{\Lambda_{n}}$, if $\ 0\leq\delta\leq\frac{(\#A-1)}{\#A}$, then
\begin{equation*}
\#\{\mathrm{v}\in A^{\Lambda_{n}}:d_{n}^{H}(\mathrm{v},\mathrm{w})\leq\delta n\}\leq2^{n\eta(\delta)}(\#A-1)^{n\delta}.
\end{equation*}
where~$\eta(\delta):=-\delta\log_{2}\delta-(1-\delta)\log_{2}(1-\delta)$.
\end{lemma}

\begin{proof}
See \cite[Lemma 2.1]{PS05} for the proof.
\end{proof}

Utilizing the method in \cite{PS05}, we prove the following proposition.

\begin{proposition}\label{ar1}
Let $(X,f)$ be a TDS, $\nu\in \mathcal{E}(X,f)$ and $\varphi$  a continuous function on $X$. For any $\eta>0$, there exists $\delta^{*}>0$, $\epsilon^{*}>0$ so that
for each neighborhood $F$ of $\nu$ in $\mathcal{M}(X)$, there exists $n^{*}_{F,\nu}\in \mathbb{N}$ such that for any $n\geq n^{*}_{F,\nu}$, there exists  a $(\delta^{*},n,\epsilon^{*})$-separated set $\Gamma_{n}\subset X_{n,F}$, such that
\begin{eqnarray*}
\sum_{x\in \Gamma_{n}}e^{S_n\varphi(x)}\geq \exp\bigr[n(h_\nu(f)+\int \varphi d\nu-\eta)\bigr].
\end{eqnarray*}
\end{proposition}

\begin{proof}
Given a small $\eta>0$, let
\begin{eqnarray*}
K_{n}=\Big\{x\in X:\mid\frac{1}{m}S_m\varphi(x)-\int \varphi d\nu\mid<\frac{\eta}{2},\ \forall\ m\geq n\Big\}
\end{eqnarray*}
then we have $\displaystyle{\lim_{n\rightarrow\infty}\nu(K_{n})=1}$ by Birkhoff ergodic theorem.

If $h_\nu(f)=0$, we choose $\Gamma_{n}=\{x\}$, where $x\in K_{n}$, then the desired result follows immediately.

If $h_\nu(f)>0$, let $h^{*}:=h_\nu(f)-\frac{\eta}{2}>0$. We choose $h^{'},\ h^{''}$ satisfying $h^{*}<h^{''}<h^{'}<h_\nu(f)$, then there exists a finite partition $\xi=\{A_{1},A_{2},\cdots,A_{k}\}$ of $X$ such that $\displaystyle{h_\nu(f,\xi):=\lim_{n\to\infty}\frac 1 n H(\nu,\xi^n)>h^{'}}$, where $\xi^n:=\bigvee_{i=0}^{n-1}f^{-i}\xi$.
Define $\phi_{n}: X\rightarrow \{1,\cdots,k\}^{\Lambda_{n}}$, $x\mapsto(\mathrm{w}_{0}\mathrm{w}_{1}\cdots\mathrm{w}_{n-1})$,
where  $
\mathrm{w}_{i}=j\   \text{if}~f^{i}(x)\in A_{j}, \ i\in\Lambda_n.
$
In the same way to define $\phi:X\rightarrow \{1,\cdots,k\}^{\mathbb{N}}$ , $x\mapsto(\omega_{0}\omega_{1}\cdots\omega_{n-1}\cdots)$.
Let\begin{eqnarray*}
Y_{n}=\phi_{n}(X), \ \ Y=\phi(X).
\end{eqnarray*}
Define $\bar{\nu}=\nu\circ\phi^{-1}$, note that for $\mathrm{w}=(\mathrm{w}_{0}\mathrm{w}_{1}\cdots\mathrm{w}_{n-1})\in Y_{n}$,
$$\bar{\nu}(\mathrm{w})=\bar{\nu}\{\omega\in Y: \omega_{i}=\mathrm{w}_{i},\ i\in\Lambda_{n}\}=\nu\Big(\bigcap_{i\in\Lambda_{n}}f^{-i}A_{\mathrm{w}_{i}}\Big).$$
Since $h_\nu(f,\xi)>h^{'}$,  there exists $n_{\xi}$ such that whenever $n\geq n_{\xi}$,
\begin{eqnarray*}
H(\nu,\xi^{n})= -\sum_{\mathrm{w}\in Y_{n}}\bar{\nu}(\mathrm{w})\log\bar{\nu}(\mathrm{w})=-\sum_{P\in \xi^n}\nu(P)\log\nu(P)>nh^{'}.
\end{eqnarray*}
Since $\nu$ is regular, for a small number $\delta>0$, there exists compact subsets $B_{j}\subset A_{j}$ $(j=1,\cdots,k)$ so that
$$\nu(A_{j}\setminus B_{j})<\frac{\delta}{2k}.$$
Let $B=\displaystyle\bigcup^{k}_{j=1}B_{j}$, then $\nu(B)>1-\frac{\delta}{2}$.

Let $F\subset \mathcal{M}(X)$ be a neighborhood of $\nu$, and let
\begin{equation}\label{eq2}
X_{n,F}^{B}=X_{n,F}\bigcap\Big\{x\in X: \frac 1m\sum_{i=0}^{m-1} I_{B}(f^ix)>1-\delta,~\forall m\ge n\Big\}\bigcap K_{n}
\end{equation}
where $I_{B}$ denotes the indicator function of $B$. Since $B$ is closed, $I_{B}$ is upper semi-continuous.
By the Birkhoff Ergodic theorem, we have that
$\displaystyle{\lim_{n\rightarrow\infty}\nu(X^{B}_{n,F})=1}.$
Hence, \begin{eqnarray*}
\nu(X\setminus X^{B}_{n,F})<\frac{h^{'}-h^{''}}{\log k}-\frac{\log2}{n\log k}
\end{eqnarray*}
for all sufficiently large $n$.
We define $\widehat{\nu}_{n,\delta}$ so that for each $\mathrm{w}\in Y_{n}$, \begin{eqnarray*}
\widehat{\nu}_{n,\delta}(\mathrm{w})=\nu\Big(\bigcap_{i\in\Lambda_{n}}f^{-i}A_{\mathrm{w}_{i}}\mid X^{B}_{n,F}\Big)=\frac{\nu\Big(\displaystyle\bigcap_{i\in\Lambda_{n}}f^{-i}A_{\mathrm{w}_{i}}\cap X^{B}_{n,F}\Big)}{\nu(X^{B}_{n,F})}.
\end{eqnarray*}
Note that for any $\mu\in \mathcal{M}(X)$, we have that $H(\mu, \xi^n)\leq n\log k$.
Therefore,  there exists $n_{F,\delta}\geq n_{\xi}$ such that for each $n\geq n_{F,\delta}$ we have
\begin{equation}\label{2}
\begin{aligned}
&\log\#\{\mathrm{w}\in Y_{n}: \widehat{\nu}_{n,\delta}(\mathrm{w})>0\}\\
&\geq H(\nu(\cdot\mid X^{B}_{n,F}), \xi^n)\\
&=-\sum_{\mathrm{w}\in Y_{n}}\widehat{\nu}_{n,\delta}(\mathrm{w})\log\widehat{\nu}_{n,\delta}(\mathrm{w})
\\&\geq\frac{H(\nu,\xi^n)-\log2-\nu(X\setminus X_{n,F}^{B})H(\nu(\cdot\mid X\setminus X_{n,F}^{B}), \xi^n)}{\nu(X_{n,F}^{B})}
\\&\geq nh'-\log2-(\frac{h^{'}-h^{''}}{\log k}-\frac{\log2}{n\log k})\cdot n\log k = nh'',
\end{aligned}
\end{equation}
where the second inequality follows from Lemma \ref{Le2}.
Put\begin{equation*}
\widetilde{Y}_n=\{\mathrm{w}\in Y_{n}\mid\widehat{\nu}_{n,\delta}(\mathrm{w})>0\}.
\end{equation*}
For any $\mathrm{w}\in\widetilde{Y}_{n}$, take a point $x_{n,\mathrm{w}}\in\phi^{-1}_{n}(\mathrm{w})$ such that $x_{n,\mathrm{w}}\in X_{n,F}^{B}$.
Let $\Xi_{n}$ denote the set which consists of all of  these points $x_{n,\mathrm{w}}$ chosen in this way. Obviously, $\Xi_{n}\subset X_{n,F}^{B}$.
By \eqref{2} and the construction of $\Xi_{n}$, we have that  \begin{equation}\label{eq1}
\phi_{n}(\Xi_{n})=\widetilde{Y}_{n}\  \text{and}~\# \Xi_{n}=\# \widetilde{Y}_{n}>e^{nh''}.
\end{equation}
Let $\Gamma_{n}\subset\Xi_{n}$ be a set of maximal cardinality satisfying that
\begin{equation}\label{ham-sep1}
x\neq x'\in\Gamma_{n}\Longrightarrow d^{H}_{n}(\phi_{n}(x),\phi_{n}(x'))>3\delta n.
\end{equation}
Since $\{B_{j}: j=1,2,\cdots,k\}$ are mutually disjoint compact subsets, there exists $\epsilon_{\delta}>0$ such that\begin{equation*}
d(x,x')>\epsilon_{\delta}~~\text{for any}~~x\in B_{i}, \ x'\in B_{j},\  i\neq j.
\end{equation*}
If $x\neq x'\in\Gamma_{n}$, it follows from \eqref{eq2} that\begin{equation}\label{ham-sep2}
\#\{i\in \Lambda_{n}: T^{i}x\notin B~\text{or}~T^{i}x'\notin B\}\leq2\delta n.
\end{equation}
Using \eqref{ham-sep1} and \eqref{ham-sep2}, for $x\neq x'\in\Gamma_{n}$ we have that  \begin{equation*}
\#\{i\in\Lambda_{n}: T^{i}x,T^{i}x'\in B~ \text{and}~d(T^{i}x,T^{i}x')>\epsilon_{\delta}\}\geq3\delta n-2\delta n=\delta n.
\end{equation*}
This implies that $\Gamma_{n}$ is  $(\delta,n,\epsilon_{\delta})$-separated.
For each $x\in\Xi_{n}$, by the maximality of $\Gamma_{n}$, there exists $x'\in\Gamma_{n}$ such that
$$d^{H}_{n}(\phi_{n}(x),\phi_{n}(x'))\leq3\delta n.$$
Hence, for $n\geq n_{F,\delta}$, by Lemma \ref{Le1} and \eqref{eq1} we have that \begin{equation*}
\#\Gamma_{n}\geq\frac{e^{nh''}}{2^{\eta(3\delta)n}(k-1)^{3\delta n}}
\end{equation*}
where $\eta(\delta)=-\delta\log_{2}\delta-(1-\delta)\log_{2}(1-\delta).$
Since $\displaystyle\lim_{\delta\rightarrow0}\eta(\delta)=0$, we may choose $\delta^{*}$ such that
$$\eta(3\delta^{*})\log2+3\delta^{*}\log(k-1)< h''-h^{*}.$$
Let $n^{*}_{F,\nu}=n_{F,\delta^{*}}$, $\epsilon^{*}=\epsilon_{\delta^{*}}$, then $\#\Gamma_{n}\geq e^{n(h_\nu(f)-\frac{\eta}{2})}$. Hence,
$$\sum_{x\in\Gamma_{n}}e^{S_n\varphi(x)}\geq\#\Gamma_{n}\cdot e^{n(\int \varphi d\nu-\frac{\eta}{2})}\geq \exp\bigr[n(h_\nu(f)+\int \varphi d\nu-\eta)\bigr].$$
\end{proof}


In Proposition \ref{ar1},  the numbers $\delta^{*}$ and $\epsilon^{*}$ in  the $(\delta^{*},n,\epsilon^{*})$-separated may depend  on the ergodic measure $\nu$. However, if the dynamical system $(X,f)$ has uniform separation property, $\delta^{*}$ and $\epsilon^{*}$ can be chosen independently of the ergodic measure.

\begin{proposition}
\label{pro2}
Let $(X,f)$ be a TDS with uniform separation property, and $\varphi$ a continuous function on $X$. Then for any $\eta > 0$, there exists $\delta^{*}>0$, $\epsilon^{*}>0$ such that for any $\mu\in \mathcal{E}(X,f)$ and any small neighborhood $F\subset \mathcal{M}(X)$ of $\mu$, there exists $n^{*}_{F,\mu,\eta}\in\mathbb{N}$ such that for each $n\geq n^{*}_{F,\mu,\eta}$,
there exists a $(\delta^{*},n,\epsilon^{*})$-separated subset $\Gamma_{n}\subset X_{n,F}$ such that  \begin{equation*}
\sum_{x\in\Gamma_{n}}e^{S_n\varphi(x)}\geq \exp\bigr[n(h_\mu(f)+\int \varphi d\mu-\eta)\bigr].
\end{equation*}
\end{proposition}

\begin{proof}
Given $\eta>0$, since the system $(X,f)$ has the uniform separation property, there exists $\delta^{*}>0$, $\epsilon^{*}>0$ such that for each $\mu\in\mathcal{ E}(X,f)$ and any neighborhood $F\subset \mathcal{M}(X)$ of $\mu$, there exists $n_{F,\mu,\eta}\in\mathbb{N}$ such that for $n\geq n_{F,\mu,\eta}$,
\begin{equation}\label{eq5}
N(F;\delta^{*},n,\epsilon^{*})\geq e^{n(h_\mu(f)-\frac{\eta}{2})}.
\end{equation}

Let $F\subset \mathcal{M}(X)$ be a small  neighborhood of $\mu\in \mathcal{E}(X,f)$ such that for any $\nu_{1},\ \nu_{2}\in F$ we have
\begin{equation}\label{eq14}
\Big|\int\varphi d\nu_{1}-\int\varphi d\nu_{2}\Big|\leq\frac{\eta}{2}.
\end{equation}
Let $\Gamma_{n}\subset X_{n,F}$ be $(\delta^{*},n,\epsilon^{*})$-separated with  maximal cardinality, by \eqref{eq5} we have that
\begin{equation*}
\#\Gamma_{n}\geq e^{n(h_\mu(f)-\frac{\eta}{2})}.
\end{equation*}
For each $x\in\Gamma_{n}$,  since $\mathcal{E}_n(x), \mu\in F$, it follows from \eqref{eq14} that \begin{equation}\label{eq13}
\Big|\frac{1}{n}S_{n}\varphi (x)-\int\varphi d\mu\Big|\leq\frac{\eta}{2}.
\end{equation}
Let $n^{*}_{F,\mu,\eta}=n_{F,\mu,\eta}$, for any $n\geq n^{*}_{F,\mu,\eta}$ we have that
\begin{equation*}
\sum_{x\in\Gamma_{n}}e^{S_n\varphi(x)}\geq\#\Gamma_{n}\cdot e^{n(\int \varphi d\mu-\frac{\eta}{2})}\ge \exp \bigr[n(h_\mu(f)+\int \varphi d\mu-\eta)\bigr].
\end{equation*}
\end{proof}

\begin{corollary}
\label{cor1}
Assume that $(X,f)$ has the uniform separation property and the ergodic measures are entropy dense, $\varphi$ is a continuous function on $X$. For any $\eta>0$, there exists $\delta^{*} > 0$, $\epsilon^{*} > 0$ such that for any $\mu\in \mathcal{M}(X,f)$ and any small neighborhood $F\subset \mathcal{M}(X)$ of $\mu$, there exists $n^{*}_{F,\mu,\eta}\in\mathbb{N}$ such that for each $n\geq n^{*}_{F,\mu,\eta}$, there exists a $(\delta^{*},n,\epsilon^{*})$-separated subset $\Gamma_{n} \subset X_{n,F}$ such that
$$\sum_{x\in\Gamma_{n}}e^{S_n\varphi(x)}\geq \exp \bigr[n(h_\mu(f)+\int \varphi d\mu-\eta)\bigr].$$
\end{corollary}

\begin{proof}
Given $\mu\in \mathcal{M}(X,f)$, if $\mu$ is ergodic, the statement follows immediately by Proposition \ref{pro2}.
If $\mu$ is not ergodic, let $F\subset \mathcal{M}(X)$ be a small  neighborhood of $\mu$. For any $\eta>0$, since the ergodic measures are entropy dense, there exists $\nu\in \mathcal{E}(X,f)\bigcap F$ such that
\begin{equation}
h_\nu(f)\geq h_\mu(f)-\frac{\eta}{3}.
\end{equation}
Since $\mu\mapsto
\int \varphi d\mu$ is continuous, we have that  \begin{equation*}
\Big | \int \varphi d \mu-\int \varphi d \nu\Big |\leq\frac{\eta}{3}.
\end{equation*}
For $\frac{\eta}{3}>0$ and the ergodic measure $\nu$, by Proposition $\ref{pro2}$ there exists $\epsilon^{*}>0$, $\delta^{*}>0$ and $n^{*}_{F,\nu,\frac{\eta}{3}}\in\mathbb{N}$ such that for each $n\geq n^{*}_{F,\nu,\frac{\eta}{3}}$, there exists a $(\delta^{*},n,\epsilon^{*})$-separated subset $\Gamma_{n} \subset X_{n,F}$ such that
\begin{equation*}
\sum_{x\in\Gamma_{n}}e^{S_n\varphi(x)}\geq \exp\bigr[n(h_\nu(f) + \int \varphi d \nu - \frac{\eta}{3})\bigr]\geq \exp \bigr[n(h_\mu(f) + \int \varphi d\mu-\eta)\bigr].
\end{equation*}
Take $n^{*}_{F,\mu,\eta} = n^{*}_{F,\nu,\frac{\eta}{3}}$, the desired result follows.
\end{proof}

The following result is one of the key ingredient in proving the variational principle of topological pressure, see \cite{wal82} for details.

\begin{lemma}
\label{le5} Let $(X,f)$ be a TDS and $\varphi$ a continuous function on $X$. Assume that
 $\{E_{n}\}_{n\geq1}$ is a sequence of $(n,\epsilon)$-separated subsets and define $$\sigma_{n}=\frac{\displaystyle\sum_{y\in E_{n}}e^{S_n\varphi(y)}\delta_{y}}{\displaystyle\sum_{z\in E_{n}}e^{S_n\varphi(z)}},\ \ \  \mu_{n}=\frac{1}{n}\sum^{n-1}_{i=0}\sigma_{_{n}}\circ f^{-i}.$$
If $\mu$ is weak-star limit point of $\{\mu_{n}\}_{n\ge 1}$, then $\mu\in \mathcal{M}(X,f)$ and
\begin{equation*}
\limsup_{n\rightarrow\infty}\frac{1}{n}\log\sum_{y\in E_{n}}e^{S_n\varphi(y)}\leq h_\mu(f)+\int \varphi d\mu.
\end{equation*}
\end{lemma}

\begin{proof}
See \cite[Theorem 9.10]{wal82} for the proof.
\end{proof}

\section{Main results}\label{main-result}

In this section, we will give the statements and proofs of the main results in this paper. Our first result shows that the measure theoretic pressure of an invariant measure is the essential supremum of point-wise measure theoretic pressure. Using the first result, one can further show that the measure theoretic pressure of an invariant measure is the essential supremum of the free energies of the measures in an ergodic decomposition. Meanwhile, we show that the measure theoretic pressure of ergodic measures defined via separated set is equal to the free energy. Furthermore, if the dynamical system has the uniform separation property and the ergodic measures are entropy dense, then the measure theoretic pressure of invariant measures (not necessarily ergodic) defined via separated set is still equal to the free energy. Finally, we show that the Hausdorff dimension of an invariant measure supported on a $C^1$ average conformal repeller is exactly the zero of the measure theoretic pressure of the same measure. Similarly, we show that the Hausdorff dimension of an invariant measure on  hyperbolic sets is the sum of the zeros of measure theoretic pressure restricted to stable and unstable directions, provided that the diffeomorphism is volume-preserving and average conformal.

\subsection{Point-wise measure theoretic pressure and measure theoretic pressure}  We first recall a useful property relating point-wise measure theoretic pressure and topological pressure on arbitrary subsets:
\begin{enumerate}
\item[(1)] if $P_{\mu}(f,\varphi,x)\le s$ for all $x\in Z$, then $P_Z(f,\varphi)\le s$;
\item[(2)] if $P_{\mu}(f,\varphi,x)\ge s$ for all $x\in Z$ and $\mu(Z)>0$, then $P_Z(f,\varphi)\ge s$
\end{enumerate}
see \cite[Theorem A]{tcz} for proofs. See \cite{mw} for the original version of Bowen entropy.

\begin{lemma}\label{mp-dis}Let $(X,f)$ be a TDS, and $\varphi:X\to \mathbb{R}$ a continuous function on $X$ and $\mu\in \mathcal{M}(X,f)$, the following properties hold:
\begin{enumerate}
\item[(i)] if $P_{\mu}(f,\varphi,x)\ge \alpha$ for $\mu$-a.e. $x$, then $P_\mu(f,\varphi)\ge \alpha$;
\item[(ii)] if $P_{\mu}(f,\varphi,x)\le \alpha$ for all $x\in Z$, then $P_Z(f,\varphi)\le \alpha$.
\end{enumerate}
\end{lemma}
\begin{proof} See \cite[Theorem A]{tcz} for the proof of the second statement. To prove the first statement, let $Z=\{x\in X: P_{\mu}(f,\varphi,x)\ge \alpha\}$. It follows that $\mu(Z)=1$, and by (1) we have that
\[
P_Z(f,\varphi)\le \alpha.
\]
By the definition of measure theoretic pressure, we have $P_\mu(f,\varphi)\le \alpha$.
\end{proof}

\begin{proposition}\label{super-pw}Let $(X,f)$ be a TDS, and $\varphi:X\to \mathbb{R}$ a continuous function on $X$ and $\mu\in \mathcal{M}(X,f)$, then
\[
P_\mu(f,\varphi)=\text{ess} \sup\{P_\mu(f,\varphi,x):x\in X\}
\]
with the essential supremum taken with respect to $\mu$.
\end{proposition}
\begin{proof}
Let $\alpha=\text{ess} \sup\{P_\mu(f,\varphi,x):x\in X\}$ and $Z=\{x\in X: P_\mu(f,\varphi,x)\le\alpha\}$, then  it follows from Lemma \ref{mp-dis} that
\[
P_Z(f,\varphi)\le \alpha.
\]
Hence, $P_\mu(f,\varphi)\le \alpha$ since $\mu(Z)=1$.  On the other hand, given a small number $\epsilon>0$, let
\[
Z_\epsilon=\{x\in X: P_\mu(f,\varphi,x)\ge \alpha-\epsilon\}.
\]
By the definition of essential supremum, we have that $\mu(Z_\epsilon)>0$. Hence,
\[
P_{\mu|Z_\epsilon}(f,\varphi)\ge \alpha-\epsilon
\]
where $\mu|Z_\epsilon(A)=\frac{\mu(Z_\epsilon\cap A)}{\mu(Z_\epsilon)}$ for any Borel measurable set $A\subset X$. By the definition of measure theoretic pressure, we have that
\[
P_{\mu}(f,\varphi)\ge P_{\mu|Z_\epsilon}(f,\varphi).
\]
Therefore, it follows that $P_{\mu}(f,\varphi)\ge\alpha-\epsilon$. Since $\epsilon$ is arbitrary, we have that $P_{\mu}(f,\varphi)\ge\alpha$. This completes the proof of the proposition.
\end{proof}

Note that the functions $h_\mu(\cdot)$ and $\varphi^*(\cdot)$ in \eqref{local-formula} are $f$-invariant almost everywhere, using\eqref{pw-formula} and Proposition \ref{super-pw}, we obtain the following formula for the measure theoretic pressure of non-ergodic measures.

\begin{corollary}\label{mp}Let $(X,f)$ be a TDS, and $\varphi:X\to \mathbb{R}$ a continuous function on $X$ and $\mu\in \mathcal{M}(X,f)$, then
\[
P_\mu(f,\varphi)=\text{ess} \sup\{h_\mu(x)+\varphi^*(x):x\in X\}
\]
with the essential supremum taken with respect to $\mu$. If, in addition, $\mu$ is ergodic then
\[
P_\mu(f,\varphi,x)=P_\mu(f,\varphi)=h_\mu(f)+\int_X \varphi d\mu
\]
for $\mu$-almost every $x\in X$.
\end{corollary}

Consequently, we get the following formula for entropy.

\begin{corollary}\label{mp1}Let $(X,f)$ be a TDS and $\mu\in \mathcal{M}(X,f)$, then
\[
E_\mu(f)=\text{ess} \sup\{h_\mu(x):x\in X\}
\]
with the essential supremum taken with respect to $\mu$. If, in addition, $\mu$ is ergodic then
\[
h_\mu(x)=E_\mu(f)=h_\mu(f)
\]
for $\mu$-almost every $x\in X$.
\end{corollary}

\subsection{Measure theoretic pressure and ergodic decompositions}This subsection is devoted to discuss the relation between  the measure theoretic pressure of an invariant measure with its ergodic decompositions. Recall that a probability Borel measure $\tau$ on $\mathcal{M}(X,f)$ (with the weak-star topology) is an \emph{ergodic decomposition} of a measure $\mu\in \mathcal{M}(X,f)$ if $\tau(\mathcal{E}(X,f))=1$ and
\[
\int_X \varphi d\mu=\int_{\mathcal{E}(X,f)} \Big( \int_X \varphi d\nu\Big)d\tau(\nu)
\]
for every $\varphi\in L^1(\mu)$.

The following theorem gives a formula for the measure theoretic pressure of an invariant measure in terms of its ergodic decomposition.

\begin{theorem}\label{mainthm}
Let $(X,f)$ be a TDS, and $\varphi:X\to \mathbb{R}$ a continuous function on $X$ and $\mu\in \mathcal{M}(X,f)$. For any ergodic decomposition $\tau$ of $\mu$ we have
\begin{eqnarray*}
\begin{aligned}
P_\mu(f,\varphi)&=\text{ess} \sup\{P_{\nu}(f,\varphi):\nu\in \mathcal{E}(X,f)\}\\
&=\text{ess} \sup\Big\{h_{\nu}(f)+\int_X \varphi d\nu:\nu\in \mathcal{E}(X,f)\Big\}
\end{aligned}
\end{eqnarray*}
with the essential supremum taken with respect to $\tau$. Consequently, the following property hold:
\[
P_\mu(f,\varphi)\ge h_\mu(f)+\int_X \varphi d\mu.
\]
\end{theorem}
\begin{proof}The second equality in the statement is clear since $P_{\nu}(f,\varphi)=h_{\nu}(f)+\int_X \varphi d\nu$ for each $\nu\in \mathcal{E}(X,f)$. Therefore, to complete the proof of the theorem, it suffices to prove the first equality. Take a subset $Z\subset X$ with $\mu(Z)=1$, then $\nu(Z)=1$ for $\tau$-almost every $\nu\in \mathcal{E}(X,f)$. This implies that $P_Z(f,\varphi)\ge P_{\nu}(f,\varphi)$ for $\tau$-almost every $\nu\in \mathcal{E}(X,f)$, which yields that
\[
P_Z(f,\varphi)\ge \text{ess} \sup\{P_{\nu}(f,\varphi):\nu\in \mathcal{E}(X,f)\}.
\]
Since $Z$ is chosen arbitrarily, by the definition of measure theoretic pressure we have that
\[
P_\mu(f,\varphi)\ge \text{ess} \sup\{P_{\nu}(f,\varphi):\nu\in \mathcal{E}(X,f)\}.
\]

To prove the reverse inequality, let $$X_0:=\Big\{x\in X: h_{\mu}(x)~\text{and}~\varphi^*(x)~\text{in}~\eqref{local-formula}~\text{is well-defined}\Big\}.$$ By Brin-Katok's local entropy formula and Birkhoff's ergodic theorem, the subset $X_0$ is  $f$-invariant  and of full $\mu$-measure. Given a small number $\epsilon>0$ and $x\in X_0$, define
\[
 \Upsilon(x)=\Big\{y\in X_0: |h_{\mu}(y)-h_{\mu}(x)|<\epsilon~\text{and}~|\varphi^*(y)-\varphi^*(x)|<\epsilon\Big\}.
\]
Clearly, $X_0=\bigcup_{x\in X_0} \Upsilon(x)$. We can choose points $y_i\in X_0$ for $i=1,2,\cdots$  such that each $\Upsilon_i:=\Upsilon(y_i)$ has positive $\mu$-measure and $\mu(\bigcup_{i\ge 1} \Upsilon_i)=1$. Note that each $\Upsilon_i$ is $f$-invariant, since the functions $h_\mu(x)$ and $\varphi^*(x)$ are $f$-invariant. Fix $i$ and let $\mu_i$ denote the normalised restriction of $\mu$ to $\Upsilon_i$. It follows from \eqref{pt-global} and the definition of $\Upsilon_i$ that
\[
h_{\mu}(y_i)-\epsilon\le h_{\mu_i}(f|\Upsilon_i)=\frac{1}{\mu(\Upsilon_i)}\int_{\Upsilon_i} h_{\mu}(x) d\mu \le h_{\mu}(y_i)+\epsilon.
\]
Let $\mathcal{M}_i(X,f)=\{\mu\in \mathcal{M}(X,f): \mu(\Upsilon_i)=1\}$. Since $\Upsilon_i$ is $f$-invariant, there is a one-to-one correspondence between the $f$-invariant ergodic probability measure on $\Upsilon_i$ and the measures in $\mathcal{E}_i(X,f):=\mathcal{M}_i(X,f)\cap \mathcal{E}(X,f)$. Using standard arguments, one can show that $\tau(\mathcal{E}_i(X,f))>0$ and that the normalization of $\tau_i$ of $\tau|\mathcal{M}_i(X,f)$ provides and ergodic decomposition of $\mu_i$. Since
\[
 h_{\mu_i}(f|\Upsilon_i)=\int_{\mathcal{E}_i(X,f)} h_{\nu}(f) d\tau_i(\nu),
\]
there exists a set $A_i\subset \mathcal{E}_i(X,f)$ of positive $\tau_i$-measure (thus also with positive $\tau$-measure) such that $h_\nu(f)\ge h_{\mu_i}(f|\Upsilon_i)-\epsilon$ for every $\nu\in A_i$. Therefore, for each $\nu\in A_i$ and each $x\in \Upsilon_i$ we have that
\[
h_\nu(f)\ge h_{\mu_i}(f|\Upsilon_i)-\epsilon\ge h_{\mu}(y_i)-2\epsilon\ge h_{\mu}(x)-3\epsilon
\]
and
\[
\int_X \varphi d\nu =\int_{\Upsilon_i} \varphi d\nu \ge \varphi^*(y_i)-\epsilon\ge  \varphi^*(x)-2\epsilon.
\]
Hence, for each $\nu\in A_i$ we have
\[
h_\nu(f)+\int_X \varphi d\nu \ge  h_{\mu}(x)+\varphi^*(x)-5\epsilon.
\]
Since $\tau(A_i)>0$ and $\mu(\Upsilon_i)>0$, it follows from Corollary \ref{mp1} that
\begin{eqnarray*}
P_\mu(f,\varphi)&=&\text{ess} \sup\{h_\mu(x)+\varphi^*(x):x\in X_0\}\\
&\le& \text{ess} \sup\Big\{h_\mu(f)+\int_X\varphi d\nu: \nu\in \mathcal{E}(X,f)\Big\}+5\epsilon.
\end{eqnarray*}
Since $\epsilon$ is arbitrary, the desired result immediately follows.

To complete the proof of the theorem, we notice that
\[
h_{\mu}(f)+\int_X \varphi d\mu =\int_{\mathcal{E}(X,f)} \bigr(h_{\nu}(f)+\int_X \varphi d\nu\bigr) d\tau(\nu).
\]
This together with the first result yield that
\[
h_{\mu}(f)+\int_X \varphi d\mu \le P_\mu(f,\varphi).
\]
\end{proof}

Consequently, we have the following formula.

\begin{corollary}\label{mainthm-ent}
Let $(X,f)$ be a TDS and $\mu\in \mathcal{M}(X,f)$. For any ergodic decomposition $\tau$ of $\mu$ we have
\begin{eqnarray*}
E_\mu(f)=\text{ess} \sup\{h_{\nu}(f):\nu\in \mathcal{E}(X,f)\}
\end{eqnarray*}
with the essential supremum taken with respect to $\tau$. Consequently, we have that
\[
E_\mu(f)\ge h_\mu(f).
\]
\end{corollary}

\begin{remark}
The inequality in the above corollary is first established by Bowen, see \cite[Theorem 1]{bo73} for details.  Our method can be regarded as another simple proof.
\end{remark}

\begin{remark} We provide an example that the inequality in Corollary \ref{mainthm-ent} may be strict. In fact, choose two ergodic measures $\mu_1,\mu_2$ such that $h_{\mu_1}(f)>h_{\mu_2}(f)$, let $\mu=\frac1 2\mu_1+\frac1 2\mu_2$, it is clear $\mu$ is $f$-invariant and satisfies that
\begin{enumerate}
\item[(1)] $E_\mu(f)=\max\{E_{\mu_1}(f), E_{\mu_2}(f)\}=h_{\mu_1}(f)$;
\item[(2)] $h_\mu(f)=\frac 1 2(h_{\mu_1}(f)+h_{\mu_2}(f)) $.
\end{enumerate}
Hence, $E_\mu(f)>h_\mu(f)$.
\end{remark}

\subsection{Measure theoretic pressure and separated sets} In this subsection, we will show that the measure theoretic pressure of ergodic measures defined via separated sets are exactly the free energy without any additional condition on the dynamical system $(X,f)$. The same phenomena can be observed in the non-ergodic case provided that the system $(X,f)$ has the uniform separation property and the ergodic measures are entropy dense.

\begin{theorem}\label{sep-erg}Let $(X,f)$ be a TDS, and $\varphi:X\to \mathbb{R}$ a continuous function on $X$. For any $\mu\in \mathcal{E}(X,f)$ we have that
\begin{equation*}
SP_{\mu}(f,\varphi)=SP_{\mu}'(f,\varphi)=h_\mu(f)+\int \varphi d\mu.
\end{equation*}
\end{theorem}

\begin{proof}
We divide the proof into two small steps:\vskip0.2cm

{\it Step 1: $\overline{SP}_{\mu}(f,\varphi)\le h_{\mu}(f)+\int \varphi d\mu$ for any $\mu\in \mathcal{M}(X,f)$.}

If $h_\mu(f)=+\infty$, then there is nothing to prove. Let  $h_\mu(f)<\infty$. Assume  that
$$\overline{SP}_{\mu}(f,\varphi)=\lim_{\epsilon\rightarrow 0}\inf_{F\ni\mu}\limsup_{n\rightarrow\infty}\frac{1}{n}\log P(F;\varphi,n,\epsilon)> h_\mu(f)+\int \varphi d\mu.$$
There exists $\epsilon^{*} > 0$ and $\delta > 0$ such that for any $\epsilon\leq\epsilon^{*}$,
$$\inf_{F\ni\mu}\limsup_{n\rightarrow \infty}\frac{1}{n}\log P(F;\varphi,n,\epsilon) \geq h_\mu(f)+\int \varphi d\mu+2\delta.$$
 Fix a number $0<\epsilon\leq\epsilon^{*}$. There exists a decreasing sequence of convex closed neighborhoods $\{C_{n}\}_{n\geq1}$ such that $\displaystyle\bigcap_{n}C_{n}=\{\mu\}$, and
\begin{equation}\label{Upper-bound}
\limsup_{n\rightarrow\infty}\frac{1}{n}\log P(C_{n};\varphi,n,\epsilon)\geq h_\mu(f)+\int \varphi d\mu+2\delta.
\end{equation}
For each $\beta>0$ and each $n\geq1$, let $E_{n}\subset X_{n,C_{n}}$ be a maximal $(n,\epsilon)$-separated subset such that \begin{equation*}
\frac{1}{n}\log\sum_{y\in E_{n}}e^{S_n\varphi(y)}\geq\frac{1}{n}\log P(C_{n};\varphi,n,\epsilon)-\beta.
\end{equation*}
Let \begin{equation*}
\sigma_{n}=\frac{\displaystyle\sum_{y\in E_{n}}e^{S_n\varphi(y)}\delta_{y}}{\displaystyle\sum_{z\in E_{n}}e^{S_n\varphi(z)}},\ \ \  \mu_{n}=\frac{1}{n}\sum^{n-1}_{i=0}\sigma_{_{n}}\circ f^{-i}.
\end{equation*}
By the construction of $\{\mu_n\}_{n\ge1}$, we know that  $\mu_{n}\in C_{n}$ for each $n\ge 1$ and $\displaystyle\lim_{n\rightarrow \infty}\mu_{n}=\mu$. By Lemma $\ref{le5}$\begin{equation*}
\limsup_{n\rightarrow \infty}\frac{1}{n}\log P(C_{n};\varphi,n,\epsilon)-\beta \leq \limsup_{n\rightarrow \infty}\frac{1}{n}\log\sum_{y\in E_{n}}e^{\varphi_{n}(y)}\leq h_\mu(f)+\int \varphi d\mu.
\end{equation*}
Since $\beta$ is arbitrary, this yields a contradiction with \eqref{Upper-bound}. \vskip0.2cm

{\it Step 2: For any $\mu\in \mathcal{E}(X,f)$, we have that $SP_{\mu}(f,\varphi)=SP_{\mu}'(f,\varphi)=h_\mu(f)+\int \varphi d\mu$. }

Since $P(F;\varphi,\delta,n,\epsilon) \leq P(F;\varphi,n,\epsilon)$ for any $\delta\in (0,1)$ and any $\epsilon>0$, by the first step we have that
\begin{equation*}
\underline{SP}_{\mu}'(f,\varphi)\le\overline{SP}_{\mu}'(f,\varphi)\le\underline{SP}_{\mu}(f,\varphi)\leq\overline{SP}_{\mu}(f,\varphi)\leq h_\mu(f)+\int \varphi d\mu.
\end{equation*}
On the other hand, it yields from Proposition \ref{ar1} that
\begin{equation*}
h_\mu(f)+\int \varphi d\mu \leq \lim_{\epsilon\rightarrow0}\lim_{\delta\rightarrow0}\inf_{F\ni\mu}\liminf_{n\rightarrow\infty}\frac{1}{n}\log P(F;\varphi,\delta,n,\epsilon)=\underline{SP}_{\mu}'(f,\varphi).
\end{equation*}
This completes the proof of the theorem.
\end{proof}

Next, we will consider the measure theoretic pressure of non-ergodic measures defined via separated sets.

\begin{theorem}Assume that the TDS $(X,f)$ has the uniform separation property  and the ergodic measures are entropy dense, and $\varphi$ is a continuous function on $X$. Then for any $\mu\in \mathcal{M}(X,f)$,
\begin{equation*}
SP_{\mu}(f,\varphi)=SP_{\mu}'(f,\varphi)=h_\mu(f)+\int \varphi d\mu.
\end{equation*}
\end{theorem}

\begin{proof}
Since $P(F;\varphi,\delta,n,\epsilon) \leq P(F;\varphi,n,\epsilon)$ for any $\delta\in (0,1)$ and any $\epsilon>0$, we have that
\begin{equation*}
\underline{SP}_{\mu}'(f,\varphi)\le\overline{SP}_{\mu}'(f,\varphi)\le\underline{SP}_{\mu}(f,\varphi)\leq\overline{SP}_{\mu}(f,\varphi)\leq h_\mu(f)+\int \varphi d\mu.
\end{equation*}
From Corollary \ref{cor1}, we have that
\begin{equation}\label{eq8}
h_\mu(f)+\int \varphi d\mu \leq \lim_{\epsilon\rightarrow0}\lim_{\delta\rightarrow0}\inf_{F\ni\mu}\liminf_{n\rightarrow\infty}\frac{1}{n}\log P_{\mu}(F;\varphi,\delta,n,\epsilon)=\underline{SP}_{\mu}'(f,\varphi).
\end{equation}
This completes the proof of the theorem.
\end{proof}

\subsection{Applications: zero of measure theoretic pressure and Hausdorff dimension of invariant measures} Let $f:M\rightarrow M$ be a $C^1$ map on a $d$-dimensional smooth Riemannian manifold $M$, and  $J\subset M$  a compact $f$-invariant subset. We assume that $J$ is an average conformal repeller of $f$. For any invariant measure $\mu$ (not necessarily ergodic) supported on $J$,  we will show that the zero of measure theoretic pressure is equal to the Hausdorff dimension of the measure $\mu$.

We first give a formula for the Hausdoff dimension of an invariant measure in terms of its ergodic decomposition.

\begin{theorem}\label{dim-dec} Let $J$ be an average conformal repeller of a $C^1$ map $f$, and let $\mu\in \mathcal{M}(f|_J)$ be an invariant measure on $J$. For any ergodic decomposition $\tau$ of $\mu$ we have
\[
\dim_H\mu =\text{ess} \sup \{\dim_H\nu: ~\nu\in \mathcal{E}(f|_J)\}
\]
with the essential supremum taken with respect to $\tau$.
\end{theorem}
\begin{proof}Given an $f$-invariant measure $\mu$, by Proposition 4 in \cite{bw} we have that
\[
\dim_H\mu\ge \text{ess} \sup \{\dim_H\nu: ~\nu\in \mathcal{E}(f|_J)\}.
\]
On the other hand, we follow the proof of Theorem 7 in \cite{bw} to prove the reverse inequality. By \eqref{ptdim-acr} and Proposition 4 in \cite{bw} we have that
\begin{eqnarray}\label{dim-pt-dec}
\dim_H\mu= \text{ess} \sup \Big\{\frac{h_\mu(x)}{\lambda(x)}: ~x\in J\Big\}.
\end{eqnarray}
Put $\varphi(x):=\frac{1}{d}\log (|\det(D_xf)|)$, where $d=\dim M$. Since
\[
\lim_{n\to\infty}\frac 1n (\log \|D_{x}f^n\|-\log m(D_{x}f^n))=0
\]
uniformly on $J$ (see \cite[Theorem 4.2]{bch} for details), where $m(\cdot)$ denotes the minimum norm of an operator, and
\[
\log m(D_{x}f^n)\le \frac{1}{d}\log (|\det(D_xf^n)|)\le \log \|D_{x}f^n\|,
\]it yields that
\[
\lambda(x)=\lim_{n\to\infty}\frac 1n \sum_{i=0}^{n-1}\varphi(f^i(x)),~~\mu-a.e.~x\in J.
\]
Let
\[
X=\{x\in J: \lambda(x)~\text{and}~h_\mu(x)~\text{are well-defined~} \},
\]
by Brin-Katok's local entropy formula and Birkhoff's ergodic theorem, the subset $X$ is  $f$-invariant  and of full $\mu$-measure. Given a small number $\epsilon>0$ and $x\in X$, define
\[
 \Upsilon(x)=\Big\{y\in X: |h_{\mu}(y)-h_{\mu}(x)|<\epsilon~\text{and}~|\lambda(y)-\lambda(x)|<\epsilon\Big\}.
\]
The sets $\Upsilon(x)$ form a cover of $X$ and we can choose points $y_i\in X$ for $i=1,2,\cdots$ such that each $\Upsilon_i:=\Upsilon(y_i)$ has positive $\mu$-measure and $\mu(\bigcup_{i\ge 1} \Upsilon_i)=1$. Note that each $\Upsilon_i$ is $f$-invariant, since the functions $h_\mu(x)$ and $\lambda(x)$ are $f$-invariant. Fix $i$ and let $\mu_i$ denote the normalised restriction of $\mu$ to $\Upsilon_i$. It follows from \eqref{pt-global} and the definition of $\Upsilon_i$ that
\[
h_{\mu}(y_i)-\epsilon\le h_{\mu_i}(f|\Upsilon_i)=\frac{1}{\mu(\Upsilon_i)}\int_{\Upsilon_i} h_{\mu}(x) d\mu \le h_{\mu}(y_i)+\epsilon.
\]
Let $\mathcal{M}_i(f|_J)=\{\mu\in \mathcal{M}(f|_J): \mu(\Upsilon_i)=1\}$. Since $\Upsilon_i$ is $f$-invariant, there is a one-to-one correspondence between the $f$-invariant ergodic probability measure on $\Upsilon_i$ and the measures in $\mathcal{E}_i(f|_J):=\mathcal{M}_i(f|_J)\cap \mathcal{E}(f|_J)$. Using standard arguments, one can show that $\tau(\mathcal{E}_i(f|_J))>0$ and that the normalization of $\tau_i$ of $\tau|\mathcal{M}_i(f|_J)$ provides and ergodic decomposition of $\mu_i$. Since
\[
 h_{\mu_i}(f|\Upsilon_i)=\int_{\mathcal{E}_i(f|_J)} h_{\nu}(f) d\tau_i(\nu),
\]
there exists a set $A_i\subset \mathcal{E}_i(f|_J)$ of positive $\tau_i$-measure (thus also with positive $\tau$-measure) such that $h_\nu(f)\ge h_{\mu_i}(f|\Upsilon_i)-\epsilon$ for every $\nu\in A_i$. Therefore, for each $\nu\in A_i$ and each $x\in \Upsilon_i$ we have that
\[
h_\nu(f)\ge h_{\mu_i}(f|\Upsilon_i)-\epsilon\ge h_{\mu}(y_i)-2\epsilon\ge h_{\mu}(x)-3\epsilon
\]
and
\[
\int_X \lambda(x) d\nu =\int_{\Upsilon_i} \lambda(x) d\nu \le \lambda(y_i)+\epsilon\le  \lambda(x)+2\epsilon.
\]
Note that $\int_X \lambda(x) d\nu$ is exactly the unique Lyapunov exponent $\lambda(\nu)$, since $J$ is an average conformal repeller.
Hence, for each $\nu\in A_i$ we have
\[
\frac{h_\mu(x)}{\lambda(x)}\le \frac{h_{\nu}(f)+3\epsilon}{\lambda(\nu)-2\epsilon}.
\]
The above observation together with \eqref{dim-acr-erg} yield that
\[
\frac{h_\mu(x)}{\lambda(x)}\le \dim_H\nu +C(\epsilon),~~\forall \nu\in A_i.
\]
Letting $\epsilon\to 0$, we have that
\[
\frac{h_\mu(x)}{\lambda(x)}\le \dim_H\nu, ~~\forall \nu\in A_i.
\]
Since $\tau(A_i)>0$ and $\mu(\Upsilon_i)>0$, it follows from \eqref{dim-pt-dec} that
\begin{eqnarray*}
\dim_H\mu\le\text{ess} \sup\{\dim_H\nu:~~\nu\in \mathcal{E}(f|_J)\}
\end{eqnarray*}
This completes the proof of the theorem.
\end{proof}

The formula in the previous theorem was first established by Barreira and Wolf for conformal repellers of a $C^{1+\alpha}$ map (see \cite{bw}). Here, we relax the smoothness to $C^1$ and extend their result for average conformal repellers which are indeed non-conformal (see \cite{zcb} for an example).

The following theorem relates the Hausdorff dimension of an invariant measure to the zero of measure theoretic pressure, which can be regarded as an Bowen's equation in the measure theoretic case.

\begin{theorem}\label{dim-expand}Let $J$ be an average conformal repeller of a $C^1$ map $f$, and let $\mu\in \mathcal{M}(f|_J)$ be an invariant measure on $J$. Then
\[
\dim_H\mu=t_0
\]
where $t_0$ is the unique root of the equation $P_\mu(f, -t\varphi)=0$, and $\varphi(x)=\frac{1}{d}\log (|\det(D_xf)|)$.
\end{theorem}
\begin{proof}Let $P(t):=P_\mu(f, -t\varphi)$, it is easy to see that the function $t\mapsto P(t)$ is continuous and strictly decreasing on the interval $[0, d]$, and $P(0)=E_\mu(f)\ge0$ and $P(d)\le 0$ by Margulis-Ruelle's inequality. Hence, there exists a unique root of the equation $P(t)=0$, say $t_0$.

Let $\tau$ be an  ergodic decomposition  of $\mu$, it follows from Theorem \ref{mainthm} that
\[
h_\nu(f)-t_0\int \varphi d\nu \le 0, ~~\tau-a.e.~\nu\in \mathcal{E}(f|_J).
\]
Since $\int \varphi d\nu =\lambda(\nu)>0$ for each $\nu\in \mathcal{E}(f|_J)$, it yields that
\[
t_0\ge \frac{h_\nu(f)}{\lambda(\nu)}=\dim_H\nu,~~~\tau-a.e.~\nu\in \mathcal{E}(f|_J).
\]
By Theorem \ref{dim-dec}, we have that
\[
t_0\ge \dim_H\mu.
\]

To prove the reverse inequality, by \eqref{dim-acr-erg} and Theorem \ref{dim-dec} we have that
\[
\dim_H\mu\ge \dim_H\nu=\frac{h_\nu(f)}{\lambda(\nu)}=\frac{h_\nu(f)}{\int \varphi d\nu},~~~\tau-a.e.~\nu\in \mathcal{E}(f|_J).
\]
This together with Corollary \ref{mp} yield that
\[
P_\nu(f, -(\dim_H\mu) \varphi )=h_\nu(f)-\dim_H\mu\int  \varphi d\nu\le 0,~~~\tau-a.e.~\nu\in \mathcal{E}(f|_J).
\]
By Theorem \ref{mainthm}, we have that
\[
P_\mu(f, -(\dim_H\mu) \varphi )\le 0.
\]
Hence,
\[
t_0\le \dim_H\mu.
\]
This completes the proof of theorem.
\end{proof}

We now consider hyperbolic diffeomorphisms and derive formula for the Hausdorff dimension of an invariant measure (not necessarily ergodic). The formulas are versions of the ones in  Theorems \ref{dim-dec} and \ref{dim-expand}. Our approach is similar to that in
Theorems \ref{dim-dec} and  \ref{dim-expand},  although it is now necessary to deal simultaneously with the stable
and unstable directions.

In the following, we will first provide a dimension formula similar to that in Theorem \ref{dim-dec} in the case of average conformal hyperbolic diffeomorphisms. Recall that $\varphi_i(x)=\frac{1}{d_i}\log \|D_xf | E^i(x)\|$ for $i=s,u$.

\begin{theorem}\label{dim-dec-ach} Let $f$ be a $C^{1}$ diffeomorphism with a compact $f$-invariant
locally maximal hyperbolic set $J$ on which $f$ is average conformal, and let $\mu$ be an
$f$-invariant  measure on $J$. For any ergodic decomposition $\tau$ of $\mu$ we have
\[
\dim_H\mu =\text{ess} \sup \{\dim_H\nu: ~\nu\in \mathcal{E}(f|_J)\}
\]
with the essential supremum taken with respect to $\tau$.
\end{theorem}
\begin{proof} The method of proof is similar to that of Theorem \ref{dim-dec}. We first note that
\[
\dim_H\mu\ge \text{ess} \sup \{\dim_H\nu: ~\nu\in \mathcal{E}(f|_J)\}.
\]

Let
\[
X=\{x\in J: \lambda_s(x),~\lambda_u(x)~\text{in \eqref{ach-lya}}~\text{and}~h_\mu(x)~\text{are well-defined~} \},
\]
 the subset $X$ is  $f$-invariant  and of full $\mu$-measure. Given a small number $\epsilon>0$ and $x\in X$, define
\[
 \Upsilon(x)=\Big\{y\in X: |h_{\mu}(y)-h_{\mu}(x)|<\epsilon~\text{and}~|\lambda_i(y)-\lambda_i(x)|<\epsilon,~~i=s,u\Big\}.
\]
Choose now points $y_i\in X$ for $i=1,2,\cdots$ such that the $f$-invariant sets
$\Upsilon_i= \Upsilon(yi)$ satisfy $\mu(\Upsilon_i) > 0$ for each $i$, and $\mu(\bigcup_{i\ge 1} \Upsilon_i)=1$. Fix $i$ and consider the normalized restriction $\mu_i$ of $\mu$ to $\Upsilon_i$. Using the same arguments as in the proof of Theorem \ref{dim-dec}, one can show that there exists a set $A_i\subset \mathcal{E}_i(f|_J)$ of positive $\tau_i$-measure (here $\tau_i$ is the normalization of $\tau|\mathcal{M}_i(f|_J)$) such that for each $\nu\in A_i$ and $x\in \Upsilon_i$ we have that
\[
h_\nu(f)\ge h_{\mu_i}(f|\Upsilon_i)-\epsilon\ge h_{\mu}(y_i)-2\epsilon\ge h_{\mu}(x)-3\epsilon
\]
and
\[
\lambda_s(\nu)=\int_{\Upsilon_i} \lambda_s(x)d\nu\ge \lambda_s(x)-2\epsilon~~\text{and}~~\lambda_u(\nu)=\int_{\Upsilon_i} \lambda_u(x)d\nu\le \lambda_u(x)+2\epsilon.
\]
Consequently, we have that
\[
h_{\mu}(x)\Big(\frac{1}{\lambda_u(x)}-\frac{1}{\lambda_s(x)}\Big)\le (h_\nu(f)+3\epsilon)\Big(\frac{1}{\lambda_u(\nu)-2\epsilon}-\frac{1}{\lambda_s(\nu)+2\epsilon}\Big)
\]
for each $\nu\in A_i$. Since $\tau(A_i)>0$, it follows from \eqref{dim-f-ach-erg} and \eqref{dim-f-ach-inv} that
\[
\dim_H\mu\le \text{ess} \sup \{\dim_H\nu: ~\nu\in \mathcal{E}(f|_J)\}+C(\epsilon)
\]
where $\epsilon\mapsto C(\epsilon)$  is a function (independent of $i$ and $\nu$) that tends to zero as $\epsilon\to 0$. This completes the proof of the theorem.
\end{proof}

The above theorem  extends Theorem 10 in \cite{bw} in twofold: first it relaxes  the smoothness of the diffeomorphism to $C^1$; on the other hand, it considers the non-conformal case in the hyperbolic setting.
\begin{theorem}\label{dim-ach-mt} Let $f$ be a $C^{1}$ volume-preserving diffeomorphism with a compact $f$-invariant locally maximal hyperbolic set $J$ on which $f$ is average conformal, and let $\mu$ be an $f$-invariant  measure on $J$. Then
\[
\dim_H\mu=t_s+t_u
\]
where $t_s$ and $t_u$ are respectively the unique root of the following equations
\[
P_\mu(f,t\varphi_s)=0~~\text{and}~~P_\mu(f,-t\varphi_u)=0.
\]
\end{theorem}
\begin{proof}The proof is similar to that of Theorem \ref{dim-expand}. It is easy to see that there exist unique roots $t_s,t_u$ of the equations $P_\mu(f,t\varphi_s)=0$ and $P_\mu(f,-t\varphi_u)=0$ respectively.

Let $\tau$ be an  ergodic decomposition  of $\mu$, it follows from Theorem \ref{mainthm} that
\[
h_\nu(f)+t_s \lambda_s(\nu) =h_\nu(f)+t_s\int \varphi_s d\nu \le 0, ~~\tau-a.e.~\nu\in \mathcal{E}(f|_J).
\]
Hence,
\[
t_s\ge -\frac{h_\nu(f)}{\lambda_s(\nu)}, ~~\tau-a.e.~\nu\in \mathcal{E}(f|_J).
\]
Similarly, one can show that
\[
t_u\ge \frac{h_\nu(f)}{\lambda_u(\nu)}, ~~\tau-a.e.~\nu\in \mathcal{E}(f|_J).
\]
The above estimates together with \eqref{dim-f-ach-erg}  yield that
\[
t_s+t_u\ge \dim_H\nu, ~~\tau-a.e.~\nu\in \mathcal{E}(f|_J).
\]
It follows from  Theorem \ref{dim-dec-ach} that
\[
t_s+t_u\ge \dim_H\mu.
\]

To prove the reverse inequality, by \eqref{dim-f-ach-erg} and Theorem \ref{dim-dec-ach} we have that
\[
\dim_H\mu\ge \dim_H\nu=h_\nu(f)\Big(\frac{1}{\lambda_u(\nu)}-\frac{1}{\lambda_s(\nu)}\Big)=h_\nu(f)
\Big(\frac{1}{\int \varphi_u d\nu}-\frac{1}{\int \varphi_s d\nu}\Big)
\]
for $\tau$-almost every $\nu\in \mathcal{E}(f|_J)$. Since $f$ is volume preserving, we have that
\[
\int \varphi_u d\nu=-\int \varphi_s d\nu,~~\forall \nu\in \mathcal{E}(f|_J).
\]
Hence,
\[
P_\nu(f, -(\frac 1 2\dim_H\mu) \varphi_u)=h_\nu(f)-\frac 1 2\dim_H\mu\int \varphi_u d\nu\le 0,~~\tau-a.e.~\nu\in \mathcal{E}(f|_J).
\]
It follows from Theorem \ref{mainthm} that
\[
P_\mu(f, -(\frac 1 2\dim_H\mu) \varphi_u)\le 0.
\]
Hence,
\[
t_u\le \frac 1 2\dim_H\mu.
\]
Similarly, one can show that
\[
t_s\le \frac 1 2\dim_H\mu.
\]
Hence,
\[
t_u+ t_s\le \dim_H\mu.
\]
This completes the proof of the theorem.
\end{proof}

\begin{question}
Whether the above theorem remains true if the diffeomorphism $f$ is not volume-preserving?
\end{question}


\subsection*{Acknowledgments} Part of the work was done when the third author visited ICTP and Fudan University, whose excellent research conditions are greatly acknowledged.


\bibliographystyle{alpha}
\bibliography{bib}

\end{document}